\documentclass[preprint]{elsarticle}

\usepackage{amsmath,amssymb,amsthm,mathrsfs}
\usepackage{lineno,hyperref}
\biboptions{sort&compress}
\usepackage{cleveref}
\modulolinenumbers[5]

\allowdisplaybreaks
\swapnumbers
\newtheorem{thm}{Theorem}[section]
\newtheorem{lem}[thm]{Lemma}
\newtheorem{prop}[thm]{Proposition}

\theoremstyle{definition}
\newtheorem{defi}[thm]{Definition}

\theoremstyle{remark}
\newtheorem{rem}[thm]{Remark}

\newcommand{\nset}{\ensuremath{\mathbb{N}}}

\newcommand{\rset}{\ensuremath{\mathbb{R}}}
\newcommand{\cset}{\ensuremath{\mathbb{C}}}

\newcommand{\supp}{\ensuremath{\mathrm{supp}}}

\DeclareMathOperator*{\conv}{conv}

\journal{arXiv.org}

\begin{document}

\begin{frontmatter}

\title{Weyl Circles for one-dimensional Moment Problems}%\tnoteref{mytitlenote}}
%\tnotetext[mytitlenote]{Working Title.} %Fully documented templates are available in the elsarticle package on \href{http://www.ctan.org/tex-archive/macros/latex/contrib/elsarticle}{CTAN}.}

\author[mymainaddress,mysecondaryaddress]{Philipp J.\ di~Dio} %\fnref{myfootnote}}
%\author[mymainaddress]{Konrad Schm\"udgen}
\address[mymainaddress]{Institute of Mathematics, University of Leipzig, Augustusplatz 10, D-04109, Germany}
\address[mysecondaryaddress]{Max Planck Institute for Mathematics in the Sciences, Inselstra{\ss}e 22, D-04103 Leipzig, Germany}

%\fntext[myfootnote]{Since 1880.}

\begin{abstract}
In the present paper we calculate the centers and radii of Weyl circles from the Nevanlinna parametrization for various one-dimensional moment problems (Hamburger moment problem, Stieltjes moment problem, truncated moment problem for finite intervals).
\end{abstract}


\end{frontmatter}

\linenumbers

%\tableofcontents

\section{Introduction}%%%
%%%%%%%%%%%%%%%%%%%%%%%%%

Let $s=(s_n)_{n\in\nset_0}$ be a real sequence and $K$ be a closed subset of the real line. The $K$-moment problem asks for a positive measure $\mu$ on $K\subseteq\rset$ such that
\[s_n = \int_K x^n~d\mu(x)\]
holds for all $n\in\nset_0$, i.e., it asks for a measure with given moments $s_n$. Necessary and sufficient conditions for the existence and uniqueness of such measures are well known and can be found in the standard literature, see e.g.\ \cite{akhiezClassical,kreinMarkovMomentProblem,schmudUnbound,karlin53,krein70,landau80,simon98} and references therein.

%\paragraph{Paper Structure}
The aim of the paper is to give a unified approach to the Weyl circles for various one-dimensional moment problems. In section 2 we derive a technical lemma which is crucial in what follows. As a first application we reprove the well known center and radius for the Hamburger moment demonstrating the best case of simplifying these formulas. Thereafter, we deal with the Stieltjes moment problem on $[a,\infty)$ showing that formulas become more complex. At last we treat the truncated moment problem on $[a,b]$, $\rset\setminus (a,b)$ ($\rset\setminus\bigcup_{i=1}^m (a_i,b_i)$) and give a simple method to handle more complicated cases.

Let us state some well-known facts which will be needed later.

\paragraph{Orthogonal polynomials of first and second kind}
Suppose we have a positive semi-definite moment sequence $s$, then define $L_s(x^n):=s_n$ on $\cset[x]$. $L_s$ induces a scalar product $\langle\,\cdot\,,\,\cdot\,\rangle_s$ on $\cset[x]$ with
\[\langle p(x),q(x)\rangle_s := L_s\left(p(x)\overline{q(x)}\right)\]
for all $p,q\in\cset[x]$. The Gram--Schmidt orthonormalization of $\{x^n\}_{n\in\nset_0}$ with respect to $L_s$ gives the orthonormal polynomials of first kind $\{P_n(x)\}_{n\in\nset_0}$ and the polynomials of second kind $Q_n(x)$ can be defined by
\[Q_n(z) := L_{s,x}\left(\frac{P_n(z) - P_n(x)}{z-x} \right)\]
where $L_{s,x}$ means that $L_s$ acts on $x$ and $z$ is treated as a constant.

\paragraph{Representing measures from self-adjoint extensions}
From $\cset[x]$ and $\langle\,\cdot\,,\,\cdot\,\rangle_s$ we can construct the Hilbert space $H_s = \overline{\cset[x]}^{\langle\,\cdot\,,\,\cdot\,\rangle_s} / \mathcal{N}$ with scalar product $\langle\,\cdot\,,\,\cdot\,\rangle$ where $\mathcal{N}$ denotes the null space of $\langle\,\cdot\,,\,\cdot\,\rangle_s$. Furthermore, defining the multiplication operator $M_x$ on $\cset[x]$ by $(M_x p)(x) := x\cdot p(x)$ we find that $M_x$ is a densely defined symmetric operator and every self-adjoint extension $X$ gives a spectral measure $\mu_X(\,\cdot\,) := \langle E_X(\,\cdot\,)1,1\rangle$ which solves the moment problem
\begin{align*}
\int_\rset x^n d\mu_X(x) &= \int_\rset x^n d\langle E(x) 1,1\rangle = \langle 1, x^n\rangle = \langle X^n 1, 1\rangle\\ &= \langle (M_x)^n 1, 1\rangle_s = \langle x^n,1\rangle_s = L_s(x^n) = s_n.
\end{align*}
On the other hand, every representing measure of the moment problem is of this form, see e.g.\ \cite[Thm.\ 16.1]{schmudUnbound}. Let $X_F$ be the Friedrichs extension of $M_x$ then we have
\[t_z := \langle 1, (X_F - z)^{-1} 1\rangle = \lim_{n\rightarrow\infty} -\frac{Q_n(z)}{P_n(z)}\]
for all $z\in\cset\setminus [\gamma,\infty)$ where $\gamma = \inf(\sigma(X_F))$, see e.g.\ \cite[Prop.\ 5.6]{simon98}.

\paragraph{Nevanlinna functions}
The solutions of the moment problem can also be characterized by using the following functions and relations.

\begin{defi}[see e.g.\ {\cite[Lem.\ 16.19 and p.\ 379]{schmudUnbound}}]\label{ABCDdef}
For $z,w\in\cset$ and $n\in\nset_0$ we define
\begin{alignat*}{3}
A_n(z,w) &:= &(z-w)\sum_{k=0}^n Q_k(z)Q_k(w) &= &\ a_n \begin{vmatrix} Q_{n+1}(z) & Q_n(z)\\ Q_{n+1}(w) & Q_n(w) \end{vmatrix}&,\\
B_n(z,w) &:= &\ -1 + (z-w)\sum_{k=0}^n P_k(z)Q_k(w) &= & a_n\begin{vmatrix} P_{n+1}(z) & P_n(z)\\ Q_{n+1}(w) & Q_n(w) \end{vmatrix}&,\\
C_n(z,w) &:= & 1 + (z-w)\sum_{k=0}^n Q_k(z)P_k(w) &=& a_n\begin{vmatrix} Q_{n+1}(z) & Q_n(z)\\ P_{n+1}(w) & P_n(w) \end{vmatrix}&,\ \text{and}\\
D_n(z,w) &:= &(z-w)\sum_{k=0}^n P_k(z)P_k(w) &=& a_n\begin{vmatrix} P_{n+1}(z) & P_n(z)\\ P_{n+1}(w) & P_n(w) \end{vmatrix}&.
\end{alignat*}
\end{defi}

As $n\rightarrow\infty$ these functions converge uniformly on each compact subset of $\rset^2$. Therefore, setting
\begin{align*}
A(z,w) &:= \lim_{n\rightarrow\infty} A_n(z,w), & B(z,w) &:= \lim_{n\rightarrow\infty} B_n(z,w),\\
C(z,w) &:= \lim_{n\rightarrow\infty} C_n(z,w), & D(z,w) &:= \lim_{n\rightarrow\infty} D_n(z,w)
\end{align*}
we obtain entire functions $A$, $B$, $C$, $D$ on $\rset^2$. For all these functions we summarize some relations in the next lemma.

\begin{lem}[see e.g.\ {\cite[p.\ 390]{schmudUnbound}}]\label{ABCDrel}
Let $z_1,z_2,z_3,z_4\in\cset$ and $n\in\nset_0\cup\{\infty\}$. Then:
\begin{enumerate}[\it i)]
\item $A_n(z_1,z_2) = - A_n(z_2,z_1)$,
\item $B_n(z_1,z_2) = - C_n(z_2,z_1)$,
\item $D_n(z_1,z_2) = - D_n(z_2,z_1)$,
\item $\left|\begin{smallmatrix}A_n(z_1,z_2) & C_n(z_1,z_4)\\ B_n(z_3,z_2) & D_n(z_3,z_4)\end{smallmatrix}\right| = C_n(z_1,z_3)\cdot C_n(z_2,z_4)$,
\item $\left|\begin{smallmatrix}A_n(z_1,z_2) & A_n(z_1,z_4)\\ A_n(z_3,z_2) & A_n(z_3,z_4)\end{smallmatrix}\right| = A_n(z_1,z_3)\cdot A_n(z_2,z_4)$
\item $\left|\begin{smallmatrix}A_n(z_1,z_2) & C_n(z_1,z_4)\\ A_n(z_3,z_2) & C_n(z_3,z_4)\end{smallmatrix}\right| = A_n(z_1,z_3)\cdot C_n(z_2,z_4)$,
\item $\left|\begin{smallmatrix}B_n(z_1,z_2) & B_n(z_1,z_4)\\ B_n(z_3,z_2) & B_n(z_3,z_4)\end{smallmatrix}\right| = D_n(z_1,z_3)\cdot A_n(z_2,z_4)$,
\item $\left|\begin{smallmatrix} B_n(z_1,z_2) & D_n(z_1,z_4)\\ B_n(z_3,z_2) & D_n(z_3,z_4)\end{smallmatrix}\right| = D_n(z_1,z_3)\cdot C_n(z_2,z_4)$, and
\item $\left|\begin{smallmatrix} D_n(z_1,z_2) & D_n(z_1,z_4)\\ D_n(z_3,z_2) & D_n(z_3,z_4)\end{smallmatrix}\right| = D_n(z_1,z_3)\cdot D_n(z_2,z_4)$.
\end{enumerate}
\end{lem}

These relations can be proven with the following lemma and \cref{ABCDdef} for $n<\infty$ and then by going to the limit $n\rightarrow\infty$ for $A$, $B$, $C$, and $D$ using uniform convergence on each compact set in $\cset^2$. It is proven by direct computation and provides hints to simplify the center and radius formulas.

\begin{lem}\label{doubledet}
For $a,b,c,d,\alpha,\beta,\gamma, \delta\in\cset$ we have
\[\begin{vmatrix}
\begin{vmatrix} a & b\\ c & d \end{vmatrix} & \begin{vmatrix} a & b\\ \gamma & \delta\end{vmatrix}\vspace{0.1cm}\\  \begin{vmatrix}\alpha &\beta\\ c&d \end{vmatrix} & \begin{vmatrix} \alpha & \beta\\ \gamma & \delta \end{vmatrix}\end{vmatrix}
= \begin{vmatrix}a & b\\ \alpha & \beta \end{vmatrix}\cdot \begin{vmatrix}c & d\\ \gamma & \delta \end{vmatrix}.\]
\end{lem}

\paragraph{Nevanlinna Parametrization}
Besides the description of all solutions by self-adjoint extensions, the solutions of the indeterminate Hamburger problem can also be characterized by Pick functions, i.e., for all $z\in\cset^+$
\[ I_\mu(z) := \int_\rset \frac{d\mu(x)}{x-z} = -\frac{C(z,0)\Phi(z) + A(z,0)}{D(z,0)\Phi(z) + B(z,0)} =: H_z^0(\Phi(z))\]
is a one-to-one correspondence between the Stieltjes transform of the representing measures $\mu$ and the M\"{o}bius transform of Pick functions $\Phi$.

\section{Main Tool}%%%
%%%%%%%%%%%%%%%%%%%%%%

Our main tool for the calculations of the center and radii of the Weyl circles is the following lemma.

\begin{lem}\label{maintool}
Let $\alpha, \beta, \gamma, \delta\in\cset$ and $I\subseteq\rset$ with $\#I \geq 3$ such that
\[w(t) = \frac{\alpha t + \beta}{\gamma t + \delta} \qquad (t\in I)\]
describes (part of) the boundary $\partial K$ of a circle $K$ in $\cset$, i.e., $w(I)\subseteq\partial K$ and therefore $\left|\begin{smallmatrix}\gamma & \delta\\ \overline{\gamma} & \overline{\delta} \end{smallmatrix}\right|\neq 0$ and $\left|\begin{smallmatrix}\alpha & \beta\\ \gamma & \delta \end{smallmatrix}\right|\neq 0$. Then the circle $K$ has center $m$ and radius $r$ given by \[m = \frac{\begin{vmatrix} \alpha & \beta\\ \overline{\gamma} & \overline{\delta}\end{vmatrix}}{\begin{vmatrix} \gamma & \delta\\ \overline{\gamma} & \overline{\delta} \end{vmatrix}}
\qquad\text{and}\qquad
r = \left| \frac{\begin{vmatrix} \alpha & \beta\\ \gamma & \delta\end{vmatrix}}{\begin{vmatrix} \gamma & \delta\\ \overline{\gamma} & \overline{\delta} \end{vmatrix}} \right|.\]
\end{lem}
\begin{proof}
Without loss of generality let $I=\rset$ and note that \[C := \left\{ \frac{1}{t+i} \;\middle|\; t\in\rset\right\} \subset \cset\] is a circle with center $M= -i/2$ and radius $R=1/2.$ Otherwise, each circle is uniquely determined by three points, i.e., only $\#I\geq 3$ is required. Then from
\begin{align*}
K &= \left\{ \frac{\alpha t + \beta}{\gamma t + \delta} \;\middle|\; t\in\rset \right\}
= \left\{ \frac{\mu}{\gamma t + \delta} \;\middle|\; t\in\rset \right\} + \frac{\alpha}{\gamma} & \left(\mu := \frac{\beta\gamma-\alpha\delta}{\gamma}\right)\\
&= \frac{\mu}{\gamma} \cdot \left\{ \frac{1}{t'+i\cdot\mathrm{Im}(\delta/\gamma)} \;\middle|\; t'\in\rset \right\} + \frac{\alpha}{\gamma} & \left( t' := t + \mathrm{Re}(\delta/\gamma)\right)\\
&= \frac{\mu}{\gamma\cdot \mathrm{Im}(\delta/\gamma)} \cdot \left\{ \frac{1}{t+i} \;\middle|\; t\in\rset \right\} + \frac{\alpha}{\gamma} & (t' := \mathrm{Im}(\delta/\gamma)t)\\
&= \frac{2i\overline{\gamma}(\beta\gamma-\alpha\delta)}{\gamma(\overline{\gamma}\delta - \gamma\overline{\delta})} \cdot C + \frac{\alpha}{\gamma} & \left( i\mathrm{Im}(\delta/\gamma) = \frac{\overline{\gamma}\delta-\gamma\overline{\delta}}{2\gamma\overline{\gamma}} \right)
\end{align*}
we find
\[m = \frac{2i\overline{\gamma}(\beta\gamma-\alpha\delta)}{\gamma(\overline{\gamma}\delta - \gamma\overline{\delta})} \cdot M + \frac{\alpha}{\gamma} = \frac{\left|\begin{smallmatrix} \alpha & \beta\\ \overline{\gamma} & \overline{\delta}\end{smallmatrix}\right|}{\left|\begin{smallmatrix} \gamma & \delta\\ \overline{\gamma} & \overline{\delta} \end{smallmatrix}\right|}
\;\text{and}\;
r = \left| \frac{2i\overline{\gamma}(\beta\gamma-\alpha\delta)}{\gamma(\overline{\gamma}\delta - \gamma\overline{\delta})} \cdot R \right| = \left| \frac{\left|\begin{smallmatrix} \alpha & \beta\\ \gamma & \delta\end{smallmatrix}\right|}{\left|\begin{smallmatrix} \gamma & \delta\\ \overline{\gamma} & \overline{\delta} \end{smallmatrix}\right|} \right|.\qedhere\]
\end{proof}

\begin{rem}
In the case that $\left|\begin{smallmatrix}\gamma & \delta\\ \overline{\gamma} & \overline{\delta} \end{smallmatrix}\right|\neq 0$ and $\left|\begin{smallmatrix}\alpha & \beta\\ \gamma & \delta \end{smallmatrix}\right| = 0$ the image of $\rset$ is a point, i.e., $m\in\cset$ and $r=0$, while in the case $\left|\begin{smallmatrix}\gamma & \delta\\ \overline{\gamma} & \overline{\delta} \end{smallmatrix}\right| = 0$ and $\left|\begin{smallmatrix}\alpha & \beta\\ \gamma & \delta \end{smallmatrix}\right|\neq 0$ the image is a line, i.e., $m=r=\infty$.
\end{rem}

\section{Hamburger Moment Problem}%%%
%%%%%%%%%%%%%%%%%%%%%%%%%%%%%%%%%%%%%

Let us now reprove the formulas for the center and radius of the Hamburger moment problem to see how the relation between $A$, $B$, $C$, and $D$ work together.

\begin{thm}[see e.g.\ {\cite[Thm.\ 16.28, Def.\ 16.7, and Lem.\ 16.32]{schmudUnbound}}]\label{hamburger}
For the one-dimensional Hamburger moment problem the Weyl circle $K_z$ has center
\begin{equation}
m_z = -\frac{C(z,\overline{z})}{D(z,\overline{z})}
\end{equation}
and radius
\begin{equation}
r_z = \frac{1}{|D(z,\overline{z})|}.
\end{equation}
\end{thm}
\begin{proof}
The boundary $\partial K_z$ is parametrized by
\[I_{\mu_t}(z) = \int_\rset \frac{d\mu_t(x)}{x-z} = -\frac{A(z,0) + t C(z,0)}{B(z,0) + t D(z,0)} \qquad(t\in\rset)\]
and therefore we have $\alpha = -C(z,0)$, $\beta = -A(z,0)$, $\gamma = D(z,0)$ and $\delta = B(z,0)$ in \cref{maintool} with the three determinants $\left|\begin{smallmatrix}\alpha & \beta\\ \gamma & \delta \end{smallmatrix}\right| = 1$, $\left|\begin{smallmatrix} \alpha & \beta\\ \overline{\gamma} & \overline{\delta}\end{smallmatrix}\right| = C(z,\overline{z})$, and $\left|\begin{smallmatrix}\gamma & \delta\\ \overline{\gamma} & \overline{\delta} \end{smallmatrix}\right| = -D(z,\overline{z})$ by \cref{ABCDrel}. This gives $m_z$ and $r_z$.
%Hence, the center $m_z$ and the radius $r_z$ are
%
%\[m_z = \frac{\begin{vmatrix} \alpha & \beta\\ \overline{\gamma} & \overline{\delta}\end{vmatrix}}{\begin{vmatrix} \gamma & \delta\\ \overline{\gamma} & \overline{\delta} \end{vmatrix}} = -\frac{C(z,\overline{z})}{D(z,\overline{z})} \qquad\text{and}\qquad r_z = \left| \frac{\begin{vmatrix} \alpha & \beta\\ \gamma & \delta\end{vmatrix}}{\begin{vmatrix} \gamma & \delta\\ \overline{\gamma} & \overline{\delta} \end{vmatrix}} \right| = \frac{1}{\left| D(z,\overline{z})\right|}.\qedhere\]
\end{proof}

\begin{rem}\label{trunchambrem}
From the parametrization
\begin{equation*}
w_n(t) = -\frac{Q_n(z)t - Q_{n+1}(z)}{P_n(z)t - P_{n+1}(z)} \tag{$t\in\rset$}
\end{equation*}
of the truncated Hamburger moment problem for $\{s_k\}_{k=0}^{2n}$ we find the center
\begin{equation}\label{mzrztrunc1}
m_z = \frac{\begin{vmatrix} -Q_n(z) & Q_{n+1}(z)\\ P_n(\overline{z}) & -P_{n+1}(\overline{z}) \end{vmatrix}}{\begin{vmatrix} P_n(z) & -P_{n+1}(z)\\ P_n(\overline{z}) & -P_{n+1}(\overline{z}) \end{vmatrix}} = -\frac{\begin{vmatrix} Q_{n+1}(z) & Q_n(z)\\ P_{n+1}(\overline{z}) & P_n(\overline{z}) \end{vmatrix}}{\begin{vmatrix} P_{n+1}(z) & P_n(z)\\ P_{n+1}(\overline{z}) & P_n(\overline{z}) \end{vmatrix}} = -\frac{C_n(z,\overline{z})}{D_n(z,\overline{z})}
\end{equation}
and the radius
\begin{equation}\label{mzrztrunc2}
r_z = \left| \frac{\begin{vmatrix} -Q_n(z) & Q_{n+1}(z)\\ P_n(z) & -P_{n+1}(z) \end{vmatrix}}{\begin{vmatrix} P_n(z) & -P_{n+1}(z)\\ P_n(\overline{z}) & -P_{n+1}(\overline{z}) \end{vmatrix}} \right| = \frac{|C_n(z,z)|}{|D_n(z,\overline{z})|} = \frac{1}{|D_n(z,\overline{z})|},
\end{equation}
see e.g.\ \cite[Thm.\ 1]{krein70}. As $n\rightarrow\infty$ the center and radius in \cref{mzrztrunc1,mzrztrunc2} tend to the center and radius of the Hamburger moment problem in \cref{hamburger}.
\end{rem}

\begin{rem}
From the previous remark and the proof of \cref{hamburger} we see that any parametrization
\[-\frac{C(z,a)t + A(z,a)}{D(z,a)t + B(z,a)}\]
gives the same center and radius for all $a\in\rset$, i.e., the change between different $a$'s results only in a linear transformation of $t$ with real coefficient in front of $t$.
\end{rem}

\section{Stieltjes Moment Problem on $[a,\infty)$}%%%
%%%%%%%%%%%%%%%%%%%%%%%%%%%%%%%%%%%%%%%%%%%%%%%%%%%%%

In the previous section we used the fact that for fixed $z\in\cset^+$ the values $\Phi(z)$ are the whole $\cset^+$ and therefore only required the image of its boundary $\partial\cset^+ = \rset \cup \{\infty\}$. The following treatment of the Stieltjes moment problem on $[a,\infty)$ reveals in a very easy way that for fixed $a\in\rset$ the boundary of $\{I_z(\mu)\}$ depends solely on $\partial\{\Phi(z)\}$. Recall the following proposition.

\begin{prop}[see e.g.\ {\cite{simon98}} or {\cite{peders97}}]\label{ainftypara}
Let $s=\{s_k\}_{k\in\nset_0}$ be an indeterminate $[a,\infty)$-moment sequence. For any $z\in\cset^+$ the formula
\[\int_a^\infty \frac{d\mu(x)}{x-z} = -\frac{C(z,a)\Phi(z) + A(z,a)}{D(z,a)\Phi(z)+B(z,a)}\]
is one-to-one correspondence between Pick functions $\Phi\in\overline{\mathfrak{P}_{a,t_a}}$ and solutions $\mu$ of this moment problem. Additionally, all $\Phi\in\overline{\mathfrak{P}_{a,t_a}}$ have the form
$\Phi(z) = \beta + \int_a^\infty \frac{d\rho(x)}{x-z}$ ($z\in\cset\setminus [a,\infty)$)
with $\beta\geq t_a$ ($>0$) and $\rho$ a measure with $\int_a^\infty \frac{d\rho(x)}{x-a+1}<\infty$ or $\Phi=\infty$.
\end{prop}

Therefore, we have to calculate $\{\Phi(z) \,|\, \Phi\in\overline{\mathfrak{P}_{a,t_a}} \}$.

\begin{lem}\label{lempickata}
Let $a\in\rset$. Then we have
\[\{\Phi(z) \,|\, \Phi\in\overline{\mathfrak{P}_{a,t_a}} \} = \left\{ \beta + \frac{\gamma}{a-z} \;\middle|\; \beta\geq t_a,\ \gamma \geq 0 \right\}\cup\{\infty\}\]
for all $z\in\cset^+$.
\end{lem}
\begin{proof}
Since every Pick function $\Phi$ in $\overline{\mathfrak{P}_{a,t_a}}$ can be written as
$\Phi(z) = \beta + \int_a^\infty \frac{d\rho(z)}{x-z}$
for some $\beta\geq t_a$ and some measure $\rho$ with $\int_a^\infty (x-a+1)^{-1}d\rho(x)<\infty$ or $\Phi=\infty$ it is sufficient to prove
\begin{multline}
\left\{\beta + \int_a^\infty \frac{d\rho(z)}{x-z} \;\middle|\; \beta\geq t_a,\  \rho\ \text{measure} \right\} = \left\{\beta + \frac{\gamma}{a-z} \;\middle|\; \beta\geq t_a,\ \gamma \geq 0 \right\}. \tag{$*$}
\end{multline}
The inclusion $\supseteq$ in ($*$) follows easily by setting $\rho = \gamma \delta_a$ ($\gamma\geq 0$). The harder part is to prove the inclusion $\subseteq$.

For $z\in\cset^+$ let
$\partial K(z) := \left\{ \frac{1}{x-z} \,\middle|\, x\in [a,\infty] \right\}$
and $K(z):= \conv \partial K(z)$. %, see \cref{scheme}.
%
%\begin{figure}[ht!]\centering
%\includegraphics[width=0.5\textwidth]{drawing.png}
%\caption{Scheme for proof of \cref{lempickata}.}\label{scheme}
%\end{figure}
%
Then by applying \cref{maintool} we see that $\partial K(z)$ is an arc of a circle with center $m$ and radius $r$ given by
$m = \frac{i}{2\mathrm{Im}(z)}$ and $r = \frac{1}{2\mathrm{Im}(z)}.$
The end points of $K(z)$ are $0$ ($x=\infty$) and $(a-z)^{-1}$ ($x=a$), i.e., the circle is divided by the line $l: \frac{\gamma}{a-z}$ ($\gamma\geq 0$) into two arcs and the disk in two parts. $\partial K(z)$ is the right arc and $K(z)$ the right part (facing to the positive real axis) since
\[ \mathrm{Re}\left(\partial_x (x-z)^{-1}\right) = \mathrm{Re}\left(\frac{-x^2 + 2x\bar{z} - \bar{z}^2}{|x-z|^4} \right) \xrightarrow{x\rightarrow\infty} -0,\]
i.e., $\partial K(z)$ is parametrized in mathematical negative direction.

Finally, let $\rho$ be a measure. Then either $\Phi(z)=\infty$ and therefore it is in the right set of ($*$) or we can find finite atomic measures $\Delta_n = \sum_{k=1}^n c_k^{(n)} \delta_{x_k^{(n)}}$ with $x_k^{(n)}\in[a,\infty)$ and $c_k^{(n)}> 0$ for all $k=1,...,n$ and $n\in\nset$ such that
\[c\cdot K(z)\ni \int_a^\infty \frac{d\Delta_n(x)}{x-z} \xrightarrow{n\rightarrow\infty} \int_a^\infty \frac{d\rho(x)}{x-z}\]
for some $c > 0$. And since $K(z)$ is closed, we also have $\int_a^\infty \frac{d\rho(x)}{x-z}$ in the right set of ($*$). This shows that also $\subseteq$ holds in ($*$) and equality is proven.
\end{proof}

\begin{rem}
The previous statement can also be viewed as a consequence of a result due to F.\ Riesz (see e.g.\ \cite{riesz11} or \cite[Thm.\ 3.5]{kreinMarkovMomentProblem}): Let $u\in C([a,b],\rset^n)$. Then the closed convex hull $\overline{\conv \{u(t) \,|\, t\in[a,b]\}}$ is the set of points $c$ admitting the representation $c=\int_a^b u(t)~d\mu(t)$ where $\mu$ is a probability measure.
\end{rem}

\begin{thm}\label{thmStielt}
If $\mu$ ranges over all solutions of the Stieltjes moment problem then for a fixed $z\in\cset^+$ the points
$w = \int_a^\infty \frac{d\mu(x)}{x-z}$
fill the closed region $L(z)$, bounded by a pair of circular arcs and lying in the upper half plane, with vertices at the points
$-\frac{C(z,a)}{D(z,a)}$ and $t_z=\lim_{n\rightarrow\infty} -\frac{Q_n(z)}{P_n(z)}=\langle 1,(X_F-z)^{-1}1\rangle$
and angle equal to $\arg (1/(a-z))$.

The bounding arcs of $L(z)$ have the parametric equations
\[w_1(t) = -\frac{C(z,a)t + A(z,a)}{D(z,a)t + B(z,a)} \tag{$t\in [t_a,\infty]$}\]
and
\[w_2(t) = -\frac{C(z,a)t + (a-z)[A(z,a) + t_a C(z,a)]}{D(z,a)t + (a-z)[B(z,a) + t_a D(z,a)]} \tag{$t\in[0,\infty]$}\]
and belong to circles $K_1(z)$ and $K_2(z)$ with center $m_i$ and radii $r_i$ given by
\begin{align*}
m_1 &= -\frac{C(z,\overline{z})}{D(z,\overline{z})}, &
m_2 &= \frac{\begin{vmatrix} (a-z)[A(z,a) + t_a C(z,a)] & C(z,a)\\ (a-\bar{z})[B(\bar{z},a) + t_a D(\bar{z},a)] & D(\bar{z},a) \end{vmatrix}}{\begin{vmatrix} (a-z)[B(z,a) + t_a D(z,a)] &  D(z,a)\\ (a-\bar{z})[B(\bar{z},a) + t_a D(\bar{z},a)] & D(\bar{z},a)\end{vmatrix}},\\
r_1 &= \frac{1}{|D(z,\overline{z})|}, &
r_2 &= \left|\frac{a-z}{\begin{vmatrix} (a-z)[B(z,a) + t_a D(z,a)] & D(z,a)\\ (a-\bar{z})[B(\bar{z},a) + t_a D(\bar{z},a)] & D(\bar{z},a)\end{vmatrix}}\right|.
\end{align*}
\end{thm}
\begin{proof}
Set
$H_z^a(t) := -\frac{C(z,a)t+A(z,a)}{D(z,a)t+B(z,a)}$
and $P := \overline{\mathfrak{P}_{a,t_a}}(z) := \{ \Phi(z) \;|\; \Phi\in\overline{\mathfrak{P}_{a,t_a}}\}$ for $z\in\cset^+$, then by \cref{lempickata} we have
\[P=\overline{\mathfrak{P}_{a,t_a}}(z) = \left\{ \beta + \frac{t}{a-z} \;\middle|\; \beta\geq t_a,\ t \geq 0 \right\}\]
and boundary $\partial P = \partial_1P \cup \partial_2 P$ with
\[\partial_1 P := [t_a,\infty] \quad\text{and}\quad \partial_2 P := \left\{t_a + \frac{t}{a-z} \;\middle|\; t\in[0,\infty]\right\}.\]
%
%see \cref{scheme2}.
%
%\begin{figure}[t!]\centering
%\includegraphics[width=\textwidth]{drawing2.png}
%\caption{Scheme for proof of \cref{thmStielt}.}\label{scheme2}
%\end{figure}
%
By \cref{ainftypara} $H_z^a$ maps $\partial P$ bijective to $L(z)$ with boundary
\[\partial L(z) = H_z^a(\partial P) = H_z^a(\partial_1 P \cup \partial_2 P) = H_z^a(\partial_1 P) \cup H_z^a(\partial_2 P).\]
Therefore, $L(z)$ is bounded by the circular arcs $L(\partial_1 P)$ and $L(\partial_2 P)$ parametrized by $w_1(t) = H_z^a(t)$ with $t\in [t_a,\infty]$ and $w_2(t)=H_z^a(t_a + (a-z)^{-1}t)$ with $t\in[0,\infty]$.

The vertices of $L(z)$ are $w_1(\infty) = w_2(\infty) = -C(z,a)/D(z,a)$ and
\begin{align*}
w_1(t_a) = w_2(0)
&= -\frac{A(z,a) + t_a C(z,a)}{B(z,a) + t_a D(z,a)}
= \lim_{n\rightarrow\infty} -\frac{A_n(z,a)-\frac{Q_n(a)}{P_n(a)}C_n(z,a)}{B_n(z,a) - \frac{Q_n(a)}{P_n(a)}D_n(z,a)}\\
&= \lim_{n\rightarrow\infty} -\frac{P_n(a) \begin{vmatrix} Q_{n+1}(z) & Q_n(z)\\ Q_{n+1}(a) & Q_n(a) \end{vmatrix} - Q_n(a) \begin{vmatrix} Q_{n+1}(z) & Q_n(z) \\ P_{n+1}(a) & P_n(a) \end{vmatrix}}{P_n(a) \begin{vmatrix} P_{n+1}(z) & P_n(z)\\ Q_{n+1}(a) & Q_n(a) \end{vmatrix} - Q_n(a) \begin{vmatrix} P_{n+1}(z) & P_n(z) \\ P_{n+1}(a) & P_n(a) \end{vmatrix}}\\
&= \lim_{n\rightarrow\infty} -\frac{Q_n(z) [P_{n+1}(a) Q_n(a) - P_n(a) Q_{n+1}(a)]}{P_n(z) [P_{n+1}(a) Q_n(a) - P_n(a) Q_{n+1}(a)]}
= \lim_{n\rightarrow\infty} -\frac{Q_n(z)}{P_n(z)}\\ &= \langle 1, (X_F - z)^{-1} 1\rangle = t_z.
\end{align*}

The angles at the vertices are by symmetry equal and also equal to $\arg(1/(a-z))$, the angle between $\partial_1 P$ and $\partial_2 P$ which is preserved by $H_z^a$.

In the case of the center and radius for $w_1$ similar calculations as in \cref{hamburger} hold leading to the exact same results. For $w_2$ applying \cref{maintool} we have $\alpha = -C(z,a)$, $\beta = (a-z)[A(z,a) + t_a C(z,a)]$, $\gamma = D(z,a)$, and $\delta = (a-z) [B(z,a) + t_a D(z,a)]$ leading to $\left|\begin{smallmatrix} \alpha & \beta\\ \gamma & \delta\end{smallmatrix}\right| = a-z$, $\left|\begin{smallmatrix} \alpha & \beta\\ \overline{\gamma} & \overline{\delta}\end{smallmatrix}\right| = \left|\begin{smallmatrix} (a-z)[A(z,a) + t_a C(z,a)] & C(z,a)\\ (a-\bar{z})[B(\bar{z},a) + t_a D(\bar{z},a)] & D(\bar{z},a) \end{smallmatrix}\right|$, and $\left|\begin{smallmatrix} \gamma & \delta\\ \overline{\gamma} & \overline{\delta}\end{smallmatrix}\right| = \left|\begin{smallmatrix} D(z,a) & (a-z)[B(z,a) + t_a D(z,a)]\\ D(\bar{z},a) & (a-\bar{z})[B(\bar{z},a) + t_a D(\bar{z},a)]\end{smallmatrix}\right|$ which gives $m_2$ and $r_2$.
\end{proof}

\section{Truncated Moment Problem on $[a,b]$}%%%
%%%%%%%%%%%%%%%%%%%%%%%%%%%%%%%%%%%%%%%%%%%%%%%%

\begin{prop}[see {\cite[Cor.\ p.\ 227f.]{krein70}}]
If $\mu$ ranges over the solutions of the truncated moment problem for $s=\{s_k\}_{k=0}^m$ with $\Omega=[a,b]$ then for a fixed $z\in\cset^+$ the points
$w = \int_a^b \frac{d\mu(t)}{t-z}$
fill the closed region $L_m(z)$, bounded by a pair of circular arcs and lying in the upper half plane.\footnote{Note that in \cite{krein70} the Stieltjes transform is chosen with a different sign, i.e., $\int_a^b \frac{d\mu(t)}{z-t}$. Therefore, Krein's $L_m(z)$ lies in the lower half plane of $\cset$.}
The bounding arcs of this region $L_m(z)$ have with $t\in[0,\infty]$ the following parametric equations for $m=2n$:
\[w_1(t) = -\frac{(-1)^n C_n(z,a)t + C_n(z,b)}{(-1)^n D_n(z,a)t + D_n(z,b)},\] 
\[w_2(t) = -\frac{(z-b)(-1)^n C_n(z,a)t + (z-a)C_n(z,b)}{(z-b)(-1)^n D_n(z,a)t + (z-a)D_n(z,b)};\]
and for $m=2n+1$:
\[w_3(t) = -\frac{(-1)^n \left|\begin{smallmatrix} Q_{n+2}(z) & Q_{n+1}(z) & Q_n(z)\\ P_{n+2}(a) & P_{n+1}(a) & P_n(a)\\ P_{n+2}(b) & P_{n+1}(b) & P_n(b)\end{smallmatrix}\right|t + (z-a)Q_{n+1}(z)}{(-1)^n \left|\begin{smallmatrix} P_{n+2}(z) & P_{n+1}(z) & P_n(z)\\ P_{n+2}(a) & P_{n+1}(a) & P_n(a)\\ P_{n+2}(b) & P_{n+1}(b) & P_n(b)\end{smallmatrix}\right|t + (z-a)P_{n+1}(z)},\]
\[w_4(t) = -\frac{(-1)^n \left|\begin{smallmatrix} Q_{n+2}(z) & Q_{n+1}(z) & Q_n(z)\\ P_{n+2}(a) & P_{n+1}(a) & P_n(a)\\ P_{n+2}(b) & P_{n+1}(b) & P_n(b)\end{smallmatrix}\right|t + (z-b)Q_{n+1}(z)}{(-1)^n \left|\begin{smallmatrix} P_{n+2}(z) & P_{n+1}(z) & P_n(z)\\ P_{n+2}(a) & P_{n+1}(a) & P_n(a)\\ P_{n+2}(b) & P_{n+1}(b) & P_n(b)\end{smallmatrix}\right|t + (z-b)P_{n+1}(z)}.\]
\end{prop}

The center and radii of the bounding circles are summarized in the next theorem.

\begin{thm}
The circle $K_1(z)$ corresponding to $w_1$ has center $m_1$ and radius $r_1$ and the circle $K_2(z)$ corresponding to $w_2$ has center $m_2$ and radius $r_2$ given by
\begin{align*}
m_1 &= -\frac{C_n(z,\overline{z})}{D_n(z,\overline{z})}, &
m_2 &= -\frac{\begin{vmatrix} (z-b)C_n(z,a) & (z-a)C_n(z,b)\\ (\bar{z}-b)D_n(\bar{z},a) & (\bar{z}-a)D_n(\bar{z},b) \end{vmatrix}}{\begin{vmatrix} (z-b)D_n(z,a) & (z-a)D_n(z,b)\\ (\bar{z}-b)D_n(\bar{z},a) & (\bar{z}-a) D_n(\bar{z},b) \end{vmatrix}},\\
r_1 &= \frac{1}{\left| D_n(z,\overline{z})\right|}, \qquad\text{and} & 
r_2 &= \left|\frac{(z-a)(z-b) D_n(a,b)}{\begin{vmatrix} (z-b)D_n(z,a) & (z-a)D_n(z,b)\\ (\bar{z}-b)D_n(\bar{z},a) & (\bar{z}-a)D_n(\bar{z},b)\end{vmatrix}} \right|.
\end{align*}
\end{thm}

\begin{proof}
We always apply \cref{maintool} with different $\alpha$, $\beta$, $\gamma$, and $\delta$. For $K_1(z)$ we have $\alpha_1 = (-1)^n B_n(a,z)$, $\beta_1 = B_n(b,z)$, $\gamma_1 = (-1)^{n+1} D_n(a,z)$, and $\delta_1 = -D_n(b,z)$ by \cref{ABCDrel} and it follows that $\left|\begin{smallmatrix} \alpha_1 & \beta_1 \\ \gamma_1 & \delta_1 \end{smallmatrix}\right| = (-1)^{n+1} D_n(a,b)$, $\left|\begin{smallmatrix} \alpha_1 & \beta_1 \\ \overline{\gamma_1} & \overline{\delta_1} \end{smallmatrix}\right| = (-1)^{n+1} C_n(z,\overline{z}) D_n(a,b)$, $\left|\begin{smallmatrix} \gamma_1 & \delta_1 \\ \overline{\gamma_1} & \overline{\delta_1} \end{smallmatrix}\right| = (-1)^n D_n(z,\overline{z}) D_n(a,b)$ giving $m_1$ and $r_1$.

For $K_2(z)$ we have $\alpha_2 = (-1)^n (z-b) B_n(a,z)$, $\beta_2 = (z-a) B_n(b,z)$, $\gamma_2 = (-1)^{n+1} (z-b) D_n(a,z)$, and $\delta_2 = -(z-a) D_n(b,z)$. It follows that
$\left|\begin{smallmatrix} \alpha_2 & \beta_2 \\ \gamma_2 & \delta_2 \end{smallmatrix}\right| = (-1)^{n+1} (z-a)(z-b) D_n(a,b)$, $\left|\begin{smallmatrix} \alpha_2 & \beta_2 \\ \overline{\gamma_2} & \overline{\delta_2} \end{smallmatrix}\right|
= (-1)^{n+1} \left|\begin{smallmatrix} (z-b)C_n(z,a) & (z-a)C_n(z,b)\\ (\bar{z}-b)D_n(\bar{z},a) & (\bar{z}-a)D_n(\bar{z},b)\end{smallmatrix}\right|$, and $\left|\begin{smallmatrix} \gamma_2 & \delta_2 \\ \overline{\gamma_2} & \overline{\delta_2} \end{smallmatrix}\right| = (-1)^n \left|\begin{smallmatrix} (z-b)D_n(z,a) & (z-a)D_n(z,b)\\ (\bar{z}-b)D_n(\bar{z},a) & (\bar{z}-a)D_n(\bar{z},b)\end{smallmatrix}\right|$ which gives $m_2$ and $r_2$.
\end{proof}

\begin{rem}
We see that the circle $K_1(z)$ is just the Weyl circle of the truncated Hamburger moment problem with the moment sequence $s=(s_k)_{k=0}^{2n}$, see \cref{trunchambrem}.
\end{rem}

\begin{rem}
The center $m_3$ and $m_4$ as well as the radii $r_3$ and $r_4$ can be calculated by the same method. But the formulas are very large since no simplifications appear due to the appearance of $3\times 3$-determinants together with $1\times 1$-determinants.
\end{rem}

\begin{rem}
We treated here only the truncated moment problem on $[a,b]$. The center and radii of the full moment problem on $[a,b]$ is then given by taking the limit $n\rightarrow\infty$ in the even case. Again, also for $n\rightarrow\infty$ $K_1(z)$ is the circle for the Hamburger moment problem and $K_2(z)$ represents the restriction of the support of $\mu$ to be in $[a,b]$.
\end{rem}

\section{Truncated Moment Problem on $\rset\setminus(a,b)$}%%%
%%%%%%%%%%%%%%%%%%%%%%%%%%%%%%%%%%%%%%%%%%%%%%%%%%%%%%%%%%%%%%

\begin{defi}[see e.g.\ {\cite[p.\ 396]{kreinMarkovMomentProblem}}]
Let $a_1 < b_1 < a_2 < b_2 < ... < a_m < b_m$ and $E_m := \rset\setminus\bigcup_{i=1}^m (a_i,b_i)$  for $m\in\nset$. Define $S(E_m)$ to be the set of Pick functions (i.e., functions $f$ holomorphic on $\cset^+$ with $f(z)\in\cset^+$ for all $z\in\cset^+$) such that $f$ is holomorphic and positive in the intervals $(a_i,b_i)$ ($i=1,...,m$).
\end{defi}

\begin{prop}[see e.g.\ {\cite[Thm.\ A.8]{kreinMarkovMomentProblem}}]
A function $F$ is in class $S(E_m)$ iff it admits a multiplicative representation
\[F(z) = C\cdot \exp \left( \int_{E_m} \left(\frac{1}{t-z} - \frac{t}{1+t^2} \right)f(t)~dt \right)\]
where $C>0$ and $0\leq f(z) \leq 1$ a.e.\ on $E_m$.
\end{prop}

\begin{prop}\label{propFSE}
Let $z\in\cset^+$, $a<b$, and $E:=\rset\setminus (a,b)$. Then
\[\{F(z)\}_{F\in S(E)} = \left\{x\in\cset \,\middle|\, 0\leq\arg x \leq \arg-\frac{z-a}{z-b}\leq \pi\right\}.\] 
\end{prop}
\begin{proof}
Since $S(E_m)$ is a cone $\{F(z)\}_{F\in S(E)}$ is a cone in $\cset^+$ for any $z\in\cset^+$. Define the rays \[F_f(z) := \rset_{\geq 0} \cdot \exp \left(\int_{E_m} \left(\frac{1}{t-z} - \frac{t}{1+t^2} \right)f(t)~dt \right).\] Then we have
\[F_{cf}(z) = F_f(z)^c \quad\forall c\geq 0 \quad\text{and}\quad F_{f_1+f_2}(z) = F_{f_1}(z)\cdot F_{f_2}(z).\tag{$*$}\]

We will show that $0\leq \arg F_f(z) \leq \arg -\frac{z-a}{z-b}$ for all measurable $f:E\rightarrow [0,1].$ Let $-\infty < a_1 < b_1 \leq a < b \leq a_2 < b_2 <\infty$ and set $g:=\chi_{[a_1,b_1]} + \chi_{[a_2,b_2]}$ as a simple function, then $F_g(z) = \rset_{\geq 0}\cdot \frac{z-b_1}{z-a_1}\cdot\frac{z-b_2}{z-a_2}$ and hence
\begin{align*}
\arg F_g(z) &= \arg(z-b_1) - \arg(z-a_1) + \arg(z-b_2) - \arg(z-a_2)\\
&= B_1 - A_1 + B_2 - A_2. %,
\end{align*}
%
%see \cref{scheme3}.
%
%\begin{figure}[ht!]\centering
%\includegraphics[width=0.75\textwidth]{drawing3.png}
%\caption{Scheme for proof of \cref{propFSE}.}\label{scheme3}
%\end{figure}
%
Going to general step functions $h = \sum_{i=1}^k c_i\chi_{[a_i,b_i]}$ ($c_i\in[0,1]$ and $a_1<b_1\leq a_2 < b_2 \leq ... < b_k$) we find
\[F_h(z) = \rset_{\geq 0}\cdot\left(\frac{z-b_1}{z-a_1}\right)^{c_1} \cdot \left(\frac{z-b_2}{z-a_2}\right)^{c_2} \cdot... \cdot \left(\frac{z-b_k}{z-a_k}\right)^{c_k}\]
and
\[\arg F_h(z) = \sum_{i=1}^k c_i(\arg(z-b_i) - \arg(z-a_i)) =\sum_{i=1}^k c_i(B_i - A_i)\]
as well as $A_i < B_i$ for all $i$. Hence, \[f_1(z)\leq f_2(z)\ \text{a.e.}  \quad\Rightarrow\quad \arg F_{f_1} \leq \arg F_{f_2}.\tag{$**$}\]
From ($**$) we see that $\min_f\arg F_f(z) = 0$ with $f_{\min}=0$, i.e., $0\leq \arg F_f(z)$. To show that $\arg F_f(z) \leq \arg-\frac{z-a}{z-b}$ let $a_1\rightarrow-\infty$, $b_1\rightarrow a$, $a_2\rightarrow b$, and $b_2\rightarrow\infty$ in $g$:
\begin{align*}
\sup_f \arg F_f(z) &= \sup B_1 - \inf A_1 + \sup B_2 - \inf A_2\\
&= \arg(z-a) - 0 + \pi - \arg(z-b) = \arg -\frac{z-a}{z-b}\leq \pi.
\end{align*}
%see \cref{scheme3}.
This supremum is in fact a maximum since it is attained by $f_{\max}=1$.
\end{proof}

\begin{prop}[see e.g.\ {\cite[Thm.\ 10.3]{derkach95}\footnote{Derkach and Malamud gave the bijection with a factor $(a_n D_n(\beta,\alpha))^{-1}$ in the $\tau(z)$ term. But since this factor is positive, we can remove it.}}]\label{missinginterval}
Let $E := \rset\setminus (a,b)$ with $a<b$ and suppose that $s=(s_k)_{k=0}^{2n}$ is a sequence strictly positive on $E$. Then for all $z\in\cset^+$ the formula
\[\int_E\frac{d\mu(t)}{t-z}=-\frac{C_n(z,a)\tau(z)-C_n(z,b)}{D_n(z,a)\tau(z)-D_n(z,b)}\]
establishes a bijective correspondence between the solution $\mu$ of the moment problem $s$ with $\supp\,\mu \subseteq E$ and functions $\tau\in S(E)$.
\end{prop}

\begin{thm}
If $\mu$ ranges over all solutions of the $\rset\setminus (a,b)$-moment problem then for a fixed $z\in\cset^+$ the points
$w = \int_{\rset\setminus (a,b)} \frac{d\mu(t)}{t-z}$
fill the closed region $L_{2n}(z)$, bounded by a pair of circular arcs and lying in the upper half plane. The bounding arcs of $L_{2n}(z)$ belong to circles $K_1(z)$ and $K_2(z)$ with center
\begin{alignat*}{3}
m_1(z) &= -\frac{C_n(z,\overline{z})}{D_n(z,\overline{z})} \quad\text{and}\quad &
m_2(z) &= \frac{\begin{vmatrix}(z-a) C_n(z,a) & (z-b) C_n(z,b)\\ (\overline{z}-a) D_n(\overline{z},a) & (\overline{z}-b) D_n(\overline{z},b)\end{vmatrix}}{\begin{vmatrix}(z-a) D_n(z,a) & (z-b) D_n(z,b)\\ (\overline{z}-a) D_n(\overline{z},a) & (\overline{z}-b) D_n(\overline{z},b)\end{vmatrix}}\\
\intertext{as well as radii}
r_1(z) &= \frac{1}{|D_n(z,\overline{z})|} \quad\text{and}&
r_2(z) &= \left|\frac{(z-a)(z-b)D_n(a,b)}{\begin{vmatrix}(z-a) D_n(z,a) & (z-b) D_n(z,b)\\ (\overline{z}-a) D_n(\overline{z},a) & (\overline{z}-b) D_n(\overline{z},b)\end{vmatrix}}\right|.
\end{alignat*}
\end{thm}
\begin{proof}
The proof proceeds as previous by applying \cref{maintool}, \cref{propFSE},  and \cref{missinginterval}.
\end{proof}

\begin{rem}
We see that $L_{2n}(z)$ is the intersection of the Weyl circle for the Hamburger moment problem and another circle $K_2^i(z)$ for the (single) restriction $\supp\mu\subseteq\rset\setminus (a_i,b_i)$ with $i=1$. Hence, for $E_m = \rset\setminus\bigcup_{i=1}^m (a_i,b_i)$ we immediately find that $w=\int_{E_m}\frac{d\mu(t)}{t-z}$ ranges over all \[K_1(z)\cap K_2^1(z)\cap ... \cap K_2^m(z)\tag{$*$}\] while $\mu$ ranges over all solutions and the moment problem is solvable on $E_m$ iff ($*$) is non-empty. This is also found in the other Weyl circle cases and we therefore get an easy method to handle more complicated cases, restrictions, and combinations: take the Weyl circle $K_1(z)$ for the Hamburger moment problem and intersect it with all Weyl circles $K_2^i(z)$ of the corresponding restrictions to gain ($*$). For instance, we have not seen how we can add up supporting intervals of the measure $\mu$ but we can rewrite them as
$\bigcup_{i=1}^m [a_i,b_i]=[a_1,b_m]\setminus\bigcup_{i=1}^{m-1} (b_i,a_{i+1})$
which then can be treated by ($*$).
\end{rem}

\section*{References}

\bibliography{../bibdata}

\begin{thebibliography}{10}
\expandafter\ifx\csname url\endcsname\relax
  \def\url#1{\texttt{#1}}\fi
\expandafter\ifx\csname urlprefix\endcsname\relax\def\urlprefix{URL }\fi
\expandafter\ifx\csname href\endcsname\relax
  \def\href#1#2{#2} \def\path#1{#1}\fi

\bibitem{akhiezClassical}
N.~I. Akhiezer, The classical moment problem and some related questions in
  analysis, Oliver \& Boyd, Edinburgh, London, 1965.

\bibitem{kreinMarkovMomentProblem}
M.~G. Krein, A.~A. Nudel'man, The Markow Moment Problem and Extremal Problems,
  no.~50 in Translations of Mathematical Monographs, American Mathematical
  Society, Providence, Rhode Island, 1977.

\bibitem{schmudUnbound}
K.~Schm{\"u}dgen, Unbounded self-adjoint operators on Hilbert space, no. 265 in
  Graduate Texts in Mathematics, Springer, Dordrecht, Heidelberg, New York,
  London, 2012.

\bibitem{karlin53}
S.~Karlin, L.~S. Shapley, Geomtry of moment spaces, no.~12 in Mem. Amer. Math.
  Soc., American Mathematical Society, Providence, Rhode Island, 1953.

\bibitem{krein70}
M.~G. Krein, The description of all solutions of the truncated power moment
  problem and some problems of operator theory, Amer.\ Math.\ Soc.\ Trans. 95
  (1970) 219--234.

\bibitem{landau80}
H.~J. Landau, The classical moment problem: Hilbertian proofs, J.~Funct.\ Anal.
  38 (1980) 255--272.

\bibitem{simon98}
B.~Simon, The classical moment problem as a self-adjoint finite difference
  operator, Adv.\ Math. 137 (1998) 82--203.

\bibitem{peders97}
H.~L. Pedersen, La param\'{e}trisation de {N}evanlinna et le probl\`{e}me des
  moments de {S}tieltjes ind\'{e}termin\'{e}, Expo.\ Math. 15 (1997) 273--278.

\bibitem{riesz11}
F.~Riesz, Sur certains syst\`emes singuliers d'\'equations int\'egrales, Ann.\
  Sci.\ \'Ecole Norm.\ Sup. 3 28 (1911) 33--62.

\bibitem{derkach95}
V.~A. Derkach, M.~M. Malamud, The extension theory of hermitian operators and
  the moment problem, J.~Math.\ Sci. 73(2) (1995) 141--242.

\end{thebibliography}

\end{document}